\documentclass[12pt]{article}

\usepackage{amssymb,amsmath,latexsym,graphics, supertabular}
\usepackage[enableskew]{youngtab}
\usepackage{amsbsy}
\usepackage{amscd}
\usepackage{verbatim}
\usepackage{amsthm}
\usepackage{times}
\usepackage{mathrsfs}
\usepackage{verbatim}
\usepackage{graphicx,subfigure}
\usepackage[colorlinks]{hyperref}
\usepackage{xypic}
\usepackage{txfonts}
\usepackage{url}

\usepackage[english]{babel}
\usepackage[T1]{fontenc}
\usepackage[utf8]{inputenc}
\usepackage{tikz}

\usetikzlibrary{backgrounds,fit,decorations.pathreplacing}
\usetikzlibrary{arrows,shapes,positioning}
\usetikzlibrary{calc,through,backgrounds,chains}

\usetikzlibrary{arrows,shapes,snakes,automata,backgrounds,petri}

\newenvironment{pf*}[1]{\proof[#1]}{\endproof}
\newtheorem{Theorem}[equation]{Theorem}
\newtheorem*{Theorem2}{Theorem}
\newtheorem{Corollary}[equation]{Corollary}
\newtheorem{Lemma}[equation]{Lemma}
\newtheorem{Proposition}[equation]{Proposition}

\theoremstyle{definition}
\newtheorem{Definition}[equation]{Definition}
\newtheorem{Example}[equation]{Example}

\newtheorem{Convention}[equation]{Convention}

\theoremstyle{remark}

\newtheorem{Exercise}[equation]{Exercise}

\newtheorem{Remark}[equation]{Remark}

\numberwithin{equation}{section}
\numberwithin{figure}{section}

\newcommand{\PP}{{\mathbb P}}
\newcommand{\C}{{\mathbb C}}
\newcommand{\Z}{{\mathbb Z}}

\newcommand{\N}{{\mathbb N}}

\newcommand{\mc}[1]{\mathcal{#1}} 
\newcommand{\ms}[1]{\mathscr{#1}} 

\newcommand{\beq}{\begin{equation}}
\newcommand{\eeq}{\end{equation}}

\newcommand{\mt}[1]{\text{#1}}

\numberwithin{equation}{section}
\numberwithin{figure}{section}

\newcommand{\bdfn}{\begin{Definition}}
\newcommand{\edfn}{\end{Definition}}
\newcommand{\brem}{\begin{Remark}}
\newcommand{\erem}{\end{Remark}}
\newcommand{\benum}{\begin{enumerate}}
\newcommand{\eenum}{\end{enumerate}}
\newcommand{\bexam}{\begin{Example}}
\newcommand{\eexam}{\end{Example}}
\newcommand{\bexer}{\begin{Exercise}}
\newcommand{\eexer}{\end{Exercise}}
\newcommand{\bthm}{\begin{Theorem}}
\newcommand{\ethm}{\end{Theorem}}

\newcommand{\blem}{\begin{Lemma}}
\newcommand{\elem}{\end{Lemma}}
\newcommand{\bprop}{\begin{Proposition}}
\newcommand{\eprop}{\end{Proposition}}
\newcommand{\bcor}{\begin{Corollary}}
\newcommand{\ecor}{\end{Corollary}}

\everymath{\displaystyle}

\def\blue#1{{\color{blue}{#1}}}

\title{Weak Order on Complete Quadrics}
\author{Mahir Bilen Can\\ Michael Joyce}

\begin{document}

\maketitle

\begin{abstract}
Using an action of the Richardson-Springer monoid on involutions, we study the weak order on the variety of complete quadrics.  
Maximal chains in the poset are explicitly determined.  Applying results of Brion, our calculations describe certain cohomology 
classes in the complete flag variety.
\end{abstract}

\section{Introduction}\label{sec:intro}

Let $G$ be a semi-simple algebraic group over an algebraically closed field of characteristic $\neq 2$ and 
let $B \subseteq G$ denote a Borel subgroup.
A $G$-variety $Y$ is {\em spherical} if $Y$ has finitely many $B$-orbits. 
Among the important spherical varieties are generalized flag varieties,
reductive algebraic monoids, symmetric spaces and their equivariant embeddings.

Let $B(Y)$ denote the finite set of all closed $B$-stable subvarieties of $Y$.
For spherical varieties, $B(Y)$ consists of the closures of $B$-orbits in $Y$.
There are two geometrically natural poset structures to place on $B(Y)$:

1. The Bruhat-Chevalley order, which is defined by $Y_1 \leq_B Y_2 \Leftrightarrow Y_1 \subseteq Y_2$.
Introduced by Ehressman in \cite{Ehresman34} when $Y$ is a homogenous space,
the literature on $\leq_B$ is vast. The case of a symmetric space is investigated by Richardson and Springer in
\cite{RS90, RS94}, and the combinatorics of the special case of the variety of smooth quadrics is investigated by Incitti \cite{Incitti04}.
The Bruhat-Chevalley order on algebraic monoids is investigated by Putcha and Renner
in \cite{Renner86}, \cite{PPR97}, \cite{Putcha02}, \cite{Putcha04}.  Other recent work includes \cite{CR11} and \cite{BC}.

2. The weak order on $B(Y)$ is defined by its covering relations: $Y_1$  is covered by $Y_2$ if and only if $Y_2 = P Y_1$
for some minimal parabolic subgroup $P \subseteq G$. In this case, $Y_1$ is covered by $Y_2$
in the Bruhat-Chevalley order; hence, the weak order is weaker than $\leq_B$. This explains the nomenclature.
The weak order on a general spherical variety is studied by Knop \cite{Knop95}, Brion \cite{Brion98, Brion01} and Springer \cite{Springer04}.

Unlike Bruhat-Chevalley order, the weak order does not ``mix'' $G$-orbits, in the sense that if $Y_1 \leq Y_2$ in the weak order,
and $\mc{O}_1, \mc{O}_2$ are the dense $B$-orbits of $Y_1, Y_2$, respectively, then $\mc{O}_1$ and $\mc{O}_2$ lie in the same $G$-orbit.
Thus, it suffices to study the weak order for homogeneous varieties.  For varieties that are not homogeneous, information about the
Bruhat-Chevalley order can be recovered from the knowledge of the weak order of each of the $G$-orbits, as shown in \cite{Brion98, Brion01, Springer04}.
In fact, the weak order on symmetric spaces is already considered in \cite{RS90, RS94} as a tool to understand the Bruhat-Chevalley order.

Let $(W,S)$ denote the Weyl group of $G$, with $S$ the reflections associated to the simple positive roots of the pair $(G,B)$.
Suppose that $Y$ is a homogeneous spherical variety. Then $B(Y)$ always has a unique maximal element, namely $Y$ itself.
For $Y' \in B(Y)$, the $W$-set of $Y'$, written
$W(Y')$, consists of all $w \in W$ of length $\text{codim}(Y')$, such that for some reduced expression $w = s_1 s_2 \cdots s_l$, $s_i \in S$,
\beq\label{eqn:W-set def}
Y = P_{s_1} P_{s_2} \cdots P_{s_l} Y'.
\eeq
In that case, (\ref{eqn:W-set def}) holds for every reduced expression of $w$.
If $B(Y)$ has a unique minimal element $Y_0$, then
the maximal chains in $B(Y)$ correspond to the reduced expressions of elements in $W(Y_0)$.

The principal result of this paper is the description of the $W$-sets for the unique minimal element in $B(\mc{O})$ for each $G$-orbit
$\mc{O}$ of the variety of complete quadrics. To explain our results, we recall the structure of the variety of complete quadrics.
For details, we refer the reader to the seminal article \cite{DP83}.

Let $(G,\sigma)$ be a pair consisting of a simply-connected, semi-simple algebraic group $G$
and an involution $\sigma: G\rightarrow G$. Let $H$ denote the normalizer of the fixed subgroup $G^\sigma$ of $\sigma$.
The {\em wonderful embedding} $X$ of the pair $(G,\sigma)$ is a certain smooth projective $G$-variety containing an open
$G$-orbit isomorphic to $G/H$ whose boundary is a union of smooth $G$-stable divisors with smooth
transversal intersections.  The boundary divisors are canonically indexed by the elements of a certain subset $\varDelta$
of a root system associated to $(G,\sigma)$. Each $G$-orbit in $X$ corresponds to a subset $I\subseteq \varDelta$.
The Zariski closure of the orbit is smooth and is equal to the transverse intersection of the boundary divisors corresponding to
the elements of $I$.

The variety of complete quadrics, which we denote by $\ms{X}=\ms{X}_n$, is the wonderful embedding of the pair $(\mt{SL}_n,\sigma)$,
where $\sigma(A) = (A^{\top})^{-1}$. In this case, $\varDelta$ is the set of simple roots associated to $\mt{SL}_n$
relative to its maximal torus of diagonal matrices contained in the Borel subgroup $B\subseteq \mt{SL}_n$
of upper triangular matrices. For combinatorial purposes, it is convenient and natural to identify $\varDelta$ with the set $[n-1] = \{1, 2, \dots, n-1\}$.

Recall that a composition of $n$ is an ordered sequence $\mu=(\mu_1,\dots, \mu_k)$ of positive integers that sum to $n$.
The compositions of $n$ are in bijection with the subsets of $[n-1]$ via
\begin{equation}\label{eqn:mu to I}
\mu =(\mu_1,\dots, \mu_k) \leftrightarrow \{ \mu_1, \mu_1+\mu_2,\dots, \mu_1+\cdots + \mu_{k-1} \},
\end{equation}
yielding an equivalent parameterization of the $G$-orbits of $\ms{X}$.
The $G$-orbit associated with the composition $\mu$ is denoted by $\mc{O}_{\mu}$.
The composition $\mu = (n)$ corresponds to the $G$-orbit of smooth quadrics in $\PP^{n-1}$, while the composition
$\mu = (1,1,\dots,1)$ corresponds to the unique closed $G$-orbit, isomorphic to the variety of complete flags in $\C^n$.

Let $\mt{S}_n$ denote the set of permutations on $[n]$. We denote by $\mt{I}_n$ the set of involutions in $\mt{S}_n$.
The $B$-orbits of $\ms{X}$ lying in $\mc{O}_{(n)}$ are parametrized by $\mt{I}_n$.
More generally, the $B$-orbits in $\mc{O}_{\mu}$ are parameterized by
combinatorial objects that we call {\em $\mu$-involutions}.
Concisely, a $\mu$-involution is a permutation of the set $[n]$ written in one-line notation and partitioned into strings by $\mu$, so that each string is an involution with respect to the relative ordering of its numbers.
For example, $[26|8351|7|94]$ is a $(2,4,1,2)$-involution and the string $8351$ is equivalent to the involution $4231$.
We denote by $\mt{I}_\mu$ the set of $\mu$-involutions.
The identity $\mu$-involution, whose entries are given in the increasing order, is the representative of the dense $B$-orbit
in the $G$-orbit $\mc{O}_\mu$.

The Richardson-Springer monoid $M(\mt{S}_n)$ is the finite monoid generated by the simple transpositions $S=\{s_1,\dots, s_{n-1}\} \subset \mt{S}_n$
subject to the following relations: $s_i^2= s_i$ for all $s_i\in S$ and $s_i s_{i+1} s_{i} = s_{i+1} s_i s_{i+1}$ for $i=1,\dots, n-2$.

We define an action of $M(\mt{S}_n)$ that is the reverse of the action associated to the weak order.
To this end, let $\pi$ be a $\mu$-involution and let $s_i \in S$ be a simple transposition. Then

\begin{enumerate}
\item If $\pi^{-1}(i) > \pi^{-1}(i+1)$, then $s_i \cdot \pi = \pi$.
\item If $\pi^{-1}(i) < \pi^{-1}(i+1)$ and the values $i$ and $i+1$ occur
in different strings of $\pi$, then $s_i \cdot \pi = s_i \pi$.
\item Otherwise, if $\pi^{-1}(i) < \pi^{-1}(i+1)$ and the values
$i$ and $i+1$ occur in the same string $\alpha_j$ of $\pi$, then the action splits into two further subcases.
\begin{enumerate}
\item If $\alpha_j$ fixes $i$ and $i+1$, then $s_i \cdot \pi = [\alpha_1| \dots| \alpha_{j-1}| s_i \alpha_{j}| \alpha_{j+1}| \dots| \alpha_k]$.
\item Otherwise, $s_i \cdot \pi = [\alpha_1| \dots| \alpha_{j-1}| s_i \alpha_j s_i| \alpha_{j+1}| \dots| \alpha_k]$.
\end{enumerate}
\end{enumerate}

\bexam
Let $\mu=(3,1,2,1)$ and $\pi = [314|6|27|5]$.
\begin{eqnarray*}
s_1 \cdot \pi = [324|6|17|5], & s_4 \cdot \pi = [315|6|27|4], \\
s_2 \cdot \pi = [314|6|27|5], & s_5 \cdot \pi = [314|6|27|5], \\
s_3 \cdot \pi = [431|6|27|5], & s_6 \cdot \pi = [314|7|26|5]. \\
\end{eqnarray*}
\eexam

It is a routine matter to check that the above action, defined for simple reflections only, extends uniquely to an action of $M(\mt{S}_n)$.
Thereby, we define the {\em reverse weak order} on $\mt{I}_{\mu}$:
$$
\pi \leq \pi' \Leftrightarrow \pi' = w \cdot \pi\ \text{for some}\ w \in M(\mt{S}_n).
$$
The weak order on $\mu$-involutions is the {\em opposite} of the weak order on the corresponding $B$-orbit closures.
That is, $\pi \leq \pi' \Leftrightarrow \overline{\mc{O}}_{\pi'} \leq \overline{\mc{O}}_{\pi}$.

Next, we briefly describe our results on the maximal chains of the weak order and give an overview of
our manuscript. There are four sections including this introduction. In the next section we recollect some of the basic definitions and lemmas regarding posets and complete quadrics.

Notwithstanding that it is a special case, the $G$-orbit of smooth quadrics is the base case of what follows.
Thus, we devote Section \ref{sec:weak order involutions} to the study of the maximal chains of the reverse weak order on involutions, only.
In particular, we determine a set $\mt{D}_n$ of permutations which determines the maximal chains of the weak order on
$\mt{I}_n$.
Precisely, $\mt{D}_n$ is the set of all $w \in \mt{S}_n$ such that for each $1 \leq i \leq n/2$, $w^{-1}(n+1-i) < w^{-1}(i)$
and there does not exist $i < j < n + 1 - i$ such that $w^{-1}(n+1-i) < w^{-1}(j) < w^{-1}(i)$.
It follows that
$$
\# \mt{D}_n=(n-1)!! = (n-1)(n-3)(n-5) \cdots.
$$
The reduced expressions of the elements of $\mt{D}_n$ provide labelings for the maximal chains on $\mt{I}_n$.

In Section \ref{sec:weak order mu-strings} we extend the results of the previous section to an arbitrary $G$-orbit.  For $w \in \mt{S}_n$, let $\ell(w)$ denote the number of pairs $i<j$ such that $w(i) > w(j)$.
For a composition $\mu=(\mu_1,\mu_2,\dots,\mu_k)$ of $n$, let $\mt{S}_\mu \subseteq \mt{S}_n$ denote the parabolic subgroup
$\mt{S}_\mu = \mt{S}_{\mu_1} \times \mt{S}_{\mu_2} \times \cdots \times \mt{S}_{\mu_k}$.
The reverse weak order is a graded poset and the rank function $\ell_\mu$ on $\mt{I}_\mu$ is given by
\begin{align}\label{A:length for I_mu}
\ell_\mu (\pi) := \min_{w \in \mt{S}_\mu} \ell(w \pi) + \sum_{i=1}^k \frac{\ell(\alpha_i) + \text{exc}(\alpha_i)}{2}, \
\pi = [\alpha_1 | \alpha_2 | \dots | \alpha_k] \in \mt{I}_\mu,
\end{align}
where $\mt{exc}(w)$ is the excedance of $w$, the number of indices such that $w(i) > i$.
Note that in (\ref{A:length for I_mu}), we view $\alpha_i$ as a permutation with respect to the relative ordering of its entries.

Let $e_\mu$ denote the minimal element of $I_\mu$ in the reverse weak order.
The main object of study of our manuscript, namely, the $W$-set of a $\mu$-involution $\pi \in \mt{I}_\mu$ is
$$
W(\pi) := \{ w \in \mt{S}_n : w \cdot e_\mu = \pi \text{ and } \ell(w) = \ell_\mu(\pi) \}.
$$
Let $\mt{D}_{\mu}$ be the set consisting of permutations $w \in \mt{S}_n$ whose $i$-th string uses the alphabet consisting of
integers $j$ such that $n - \sum_{s=1}^i \mu_s < j \leq n - \sum_{s=1}^{i-1} \mu_{s-1}$ and has its associated 
permutation in $\mt{D}_{\mu_i}$.
Our second main result, which we prove in Section \ref{sec:weak order mu-strings} is the following


\begin{Theorem}\label{T:main}
Let $\pi_{0,\mu}$ be the maximal element in the reverse weak order on $\mt{I}_\mu$. Then $W(\pi_{0,\mu})=\mt{D}_\mu$.
\end{Theorem}

Note that, when $\mu = (1,1,\dots, 1)$, the $W$-set of an element $\pi \in \mt{I}_\mu \cong \mt{S}_n$ is the singleton $\{ \pi \}$. 
Therefore,  the classical weak order on $\mt{S}_n$ is a special case of the weak order on $\mu$-involutions.

Finally, in Section \ref{sec:geom}, we discuss a geometric corollary and a conjecture which merit further study.  
The theorem is an immediate consequence of Theorem \ref{T:main} and \cite[Corollary 1.3]{Brion98}.

\begin{Theorem}\label{cor:geom}
Let $\ms{X}_\mu$ be the closure of the dense $B$-orbit of $\mc{O}_\mu$, $Y_{\mu}$ the closure of the 
unique closed $B$-orbit of $\mc{O}_{\mu}$, and $i^* : H^*(\ms{X}_\mu; \Z) \rightarrow H^*(G / B^-; \Z)$ 
the restriction map on cohomology induced from the natural inclusion
$i : G / B^- \hookrightarrow \ms{X}_\mu$.  Then
$$
i^* ( [Y_{\mu}] ) = 2^{\lfloor n/2 \rfloor} \sum_{w \in \mt{D}_\mu} \left[\overline{B w^{-1} B^-/B^-}\right].
$$
\end{Theorem}

Based on numerical evidence for $n \leq 10$, we conjecture that the class $i^*([Y_{n}])$ (coming from the 
largest $G$-orbit) factors into a product of binomials.  
More specifically, we expect that
$$
i^*([Y_{n}]) = \prod_{1 \leq i \leq j \leq n-i} (x_i + x_j) = 2^{\lfloor n/2 \rfloor} 
\prod_{i=1}^{\lfloor n/2 \rfloor} x_i \prod_{1 \leq i < j < n+1-i} (x_i + x_j).
$$

\vspace{.35in}

\noindent \textbf{Acknowledgement.}
The first author is partially supported by the Louisiana Board of Regents enhancement grant.

\section{Notation and Preliminaries}

All varieties are defined over an algebraically closed field of characteristic $\neq 2$.
Throughout, $B$ denotes the Borel subgroup of upper triangular matrices in $G=\mt{SL}_n$ and $B^{-}$ 
its opposite Borel subgroup of lower triangular matrices.

\subsection{Poset terminology}

We denote the Weyl group of $\mt{SL}_n$, the symmetric group of permutations on $n$ letters, by $\mt{S}_n$.
For $w\in \mt{S}_n$, we write its cycle representation using parentheses and its one-line notation using brackets.
For example, $(3,5)$ and $[125436]$ both denote the permutation in $\mt{S}_6$ that interchanges $3$ and $5$ while fixing the other
four numbers.

All posets are assumed to be finite and assumed to have a maximal element.
Recall that a poset $P$ with a minimal element $\hat{0}$ is {\em graded} if every maximal chain in $P$ has the same length.
We denote by $\mt{rk}: P\rightarrow \N$ the {\em rank function} on $P$ so that for $x\in P$, $\mt{rk}(x)$ is the length of a
maximal chain from $\hat{0}$ to $x$.
The {\em rank} of a graded poset $P$, denoted by $\mt{rk}(P)$, is defined to be $\mt{rk}(\hat{1})$, where $\hat{1}$ is the maximal element.

When a solvable group $B$ acts on a projective variety with finitely many orbits,
the poset consisting of irreducible $B$-stable subvarieties with respect to inclusion ordering is a graded poset \cite[Exercise 8.9.12]{RennerBook}.
A well known example is the Bruhat-Chevalley ordering on the Schubert varieties (closures of the $B$-orbits in the flag variety $\mt{SL}_n/B$).
Since Schubert varieties are indexed by the permutations, $\mt{S}_n$ acquires the Bruhat-Chevalley ordering. Furthermore, the rank function
of the induced poset structure on $\mt{S}_n$ is given by the number of inversions:
\begin{align}\label{A:length on S_n}
\mt{rk}(w)= \ell(w) = \# \{ (i,j):\ i<j\ \text{and}\ w(i) > w(j) \}.
\end{align}

An involution $w\in \mt{S}_n$ is an element of order $\leq 2$. We denote by $\mt{I}_n$ the set of involutions in $\mt{S}_n$.
It is shown in \cite{RS90} that the elements of $\mt{I}_n$ are in one-to-one correspondence with $B$-orbits in the space of 
invertible symmetric $n\times n$ matrices.
Therefore, similar to the Bruhat-Chevalley ordering on $\mt{S}_n$, there is an induced Bruhat-Chevalley ordering on involutions.
It is shown in \cite{RS94} that the opposite of the Bruhat-Chevalley order on $\mt{I}_n$ agrees with the restriction of Bruhat-Chevalley order on 
$\mt{S}_n$.

We denote by $\ell_{(n)}$ the rank function on $\mt{I}_n$.
Combinatorial properties of the opposite pair $(\mt{I}_n,\leq_B^{op})$, including the following formulation of $\ell_{(n)}$ are developed in \cite{Incitti04}:
\begin{align}
\ell_{(n)}(w) = \frac{\ell(w)+\mt{exc}(w)}{2},
\end{align}
where $\mt{exc}(w)$, the excedance of $w$, is defined by (\ref{A:length for I_mu}).
Note that the excedance of an involution is the number of 2-cycles that appear in its cycle decomposition.

\begin{Remark}
Recall that the covering relations of the weak order are covering relations for the Bruhat-Chevalley order, and therefore,
a maximal chain in a weak order is a maximal chain in the Bruhat-Chevalley order, also. We conclude from this observation that
the weak order is a graded poset and its rank function agrees with that of the Bruhat-Chevalley ordering.
\end{Remark}

\subsection{Complete Quadrics}

Let $X_0$ denote the open set of the projectivization of $\text{Sym}_n$, the space of symmetric $n \times n$ matrices, 
with non-zero determinant.

A smooth quadric hypersurface $\mc{Q}$ in $\PP^{n-1}$ is the vanishing locus of a
quadratic polynomial of the form $x^\top A x$, where $A$ is a symmetric, invertible $n\times n$ matrix
and $x$ is a column vector of variables.
The correspondence $\mc{Q} \leftrightsquigarrow A$ is unique up to scaler multiples of $A$.
There is a transitive action of $\mt{SL}_n$ on smooth quadrics induced from the action on matrices:
\begin{align}\label{A:SLnaction}
g \cdot A = \theta(g) A g^{-1},
\end{align}
where $g \in \mt{SL}_n$ and $\theta$ is the involution $\theta(g) = (g^\top)^{-1}$.
Let $\mt{SO}_n \subset \mt{SL}_n$ denote the orthogonal subgroup consisting of matrices $g \in \mt{SL}_n$ such that $(g^\top)^{-1}=g$.  
The space of smooth quadric hypersurfaces $X_0$ is identified with $\mt{SL}_n / \widetilde{\mt{SO}}_n$.

The minimal wonderful embedding of $X_0$ is the variety of complete quadrics $\ms{X}$. Concretely, a point is given by specifying a flag
\begin{equation}\label{eqn:flag}
F: 0 = V_0 \subset V_1 \subset \cdots \subset V_k = \C^n
\end{equation}
and a smooth quadric in $\PP(V_i / V_{i-1})$ for each $1 \leq i \leq k$.  The action of $\mt{SL}_n$ on $X_0$ extends to an action on $\ms{X}$
as follows.

Let $(F,Q) \in \ms{X}$ be a point given by the flag $F$ and a sequence of smooth quadrics
$Q = (Q_1,\dots, Q_k)$. Suppose $A_i$, $1\leq i \leq k$ denote the corresponding invertible symmetric matrices.
Then, for $g\in \mt{SL}_n$, $g\cdot (F,Q) = (g F, g Q)$ is given by
$$
g F: 0 = g(V_0) \subset g(V_1) \subset \cdots \subset g(V_k) = \C^n
$$
and the sequence of smooth quadrics $g Q_i$ defined by the positive definite quadratic form $(g^{\top})^{-1} A_i g^{-1}$ on $g(V_i)$
for $i=1,\dots, k$.

Let $\mc{O}_{\mu}$ be a $G$-orbit of $\ms{X}$ corresponding to the composition $\mu$.
Then $\mc{O}_{\mu}$ consists of those complete quadrics whose flag in (\ref{eqn:flag}) satisfies $\dim(V_i / V_{i-1}) = \mu_i$.
We denote by $\ms{X}_{\mu}$ the closure of $\mc{O}_{\mu}$ in $\ms{X}$.
Then $\mc{O}_{\mu'} \subseteq \ms{X}_{\mu}$ if and only if the set corresponding to the composition $\mu$ is contained in the
set corresponding to that of $\mu'$.

The $B$-orbits in $\mc{O}_{\mu}$ are parameterized by the $\mu$-involutions \cite{Springer04}.
Indeed, associated to a $\mu$-involution $\pi$ is a distinguished complete quadric $Q_{\pi}$. 
Let us explain.

Viewed as a permutation, $\pi \in \mt{I}_\mu$ has the decomposition $\pi = u v$ with $u \in \mt{S}_\mu$ and $v \in \mt{S}^\mu$,
where $\mt{S}^\mu$ is the minimal length right coset representatives of the parabolic subgroup $\mt{S}_\mu$ in
$\mt{S}_n$.
Suppose $\mu =(\mu_1,\dots, \mu_k)$ and let $e_i$ denote the $i$-th standard basis vector of $\C^n$.
Then the desired flag of $Q_{\pi}$ is given by the subspaces $V_i$, $i=1,2,\dots,k$, which are spanned by 
$e_{\pi(j)}$ for $1 \leq j \leq \mu_1 + \mu_2 + \cdots + \mu_i$.
To construct the corresponding sequence of smooth quadrics, consider $(u_1, u_2, \dots, u_k)$, the image of $u$ under the isomorphism
$\mt{S}_\mu \cong \mt{S}_{\mu_1} \times \mt{S}_{\mu_2} \times \cdots \times \mt{S}_{\mu_k}$.
Since $\pi$ is a $\mu$-involution, each $u_i \in \mt{I}_{\mu_i}$.
Then the smooth quadric in $\PP(V_i / V_{i-1})$ that defines $Q_{\pi}$ is given by the symmetric matrix in the permutation matrix
representation of $u_i$.

In a recent paper \cite{CJ11}, the authors give several different combinatorial interpretations and asymptotic estimates 
for the total number of $B$-orbits in $\ms{X}$.

\section{Weak Order on Involutions}\label{sec:weak order involutions}

When specialized to the big $G$-orbit of smooth quadrics, the Richardson-Springer monoid action 
is given by the following

\bdfn\label{def:down act inv}
Let $\pi \in \mt{I}_n$ and $s_i \in S$. Then the Richardson-Springer monoid action of $s_i$ on $\pi$ is given by
$$
s_i \cdot \pi =
\begin{cases}
\pi & \text{if}\ \pi^{-1}(i+1) < \pi^{-1}(i) \\
s_i\pi & \text{if}\ \pi(i)=i\ \text{and}\  \pi(i+1)=i+1 \\
s_i \pi s_i & \text{otherwise}
\end{cases}
$$
\edfn

\bexam
Let $\pi = (15)(27) = [5734162] \in \mt{S}_7$.  Then
\begin{align*}
s_1 \cdot \pi &= (1,7)(2,5) = [7534261] && s_4 \cdot \pi = (1,5)(2,7) = [5734162] \\
s_2 \cdot \pi &= (1,5)(2,7) = [5734162] && s_5 \cdot \pi = (1,6)(2,7) = [6734512] \\
s_3 \cdot \pi &= (1,5)(2,7)(3,4) = [5743162] && s_6 \cdot \pi = (1,5)(2,7) = [5734162].
\end{align*}
\eexam

\blem\label{lem:down act basics inv}
\benum
\item The operator $\cdot$ in Definition \ref{def:down act inv} extends to an action of $M$ on the set $\mt{I}_n$ of involutions in $\mt{S}_n$.
That is\label{lem:da 1}
    \benum
    \item $s_i \cdot (s_i \cdot \pi) = s_i \cdot \pi$;\label{lem:da 1a}
    \item $s_i \cdot (s_j \cdot \pi) = s_j \cdot (s_i \cdot \pi)\,$ if $\,|j - i| > 1$;\label{lem:da 1b}
    \item $s_i \cdot (s_{i+1} \cdot (s_i \cdot \pi)) = s_{i+1} \cdot (s_i \cdot (s_{i+1} \cdot \pi))$.\label{lem:da 1c}
    \eenum
\item If $\pi^{-1}(i+1) < \pi^{-1}(i)$, then $\ell_{(n)}(s_i \cdot \pi) = \ell_{(n)}(\pi)$.  If $\pi^{-1}(i) < \pi^{-1}(i+1)$, then
$\ell_{(n)}(s_i \cdot \pi) = \ell_{(n)}(\pi) + 1$.\label{lem:da 2}
\eenum
\elem

Because of (\ref{lem:da 1b}) and (\ref{lem:da 1c}), we define $w \cdot \pi = s_{i_1} \cdot (s_{i_2} \cdot \cdots (s_{i_k} \cdot \pi))$
for any reduced expression $w = s_{i_1} s_{i_2} \cdots s_{i_k}$.  It follows from (\ref{lem:da 2}) that
$\ell_{(n)}(w \cdot \pi) \leq \ell_{(n)}(\pi) + \ell(w)$.


\bdfn\label{def:weak order inv}

The reverse weak order on $\mt{I}_n$ is defined by $\pi \leq \pi' \Leftrightarrow \pi' = w \cdot \pi$ for some $w \in M$.
It has a unique minimal element $e = e_n = [1 \, 2 \, \dots \, n]$ and a unique maximal element
$\pi_0 = \pi_{0,n} = [n \, (n-1) \, \dots \, 1] = (1,n)(2,n-1)\cdots$.

\edfn

\brem\label{rem:rank calc}
The reverse weak order is a graded poset of rank $\ell_{(n)}(\pi_{0,n})$.  A simple calculation reveals that $\ell_{(n)}(\pi_{0,n}) = \frac{n^2}{4}$ if $n$ is even and $\ell_{(n)}(\pi_{0,n}) = \frac{n^2-1}{4}$ if $n$ is odd.  In particular, $\ell_{(n)}(\pi_{0,n}) - \ell_{(n)}(\pi_{0,n-2}) = n - 1$ for all $n > 2$.
\erem


\bdfn\label{def:W-set inv}
The $W$-set of $\pi \in \mt{I}_n$ is
$$W(\pi) := \{ w \in W : w \cdot e = \pi \text{ and } \ell(w) = \ell_{(n)}(\pi) \}.$$
\edfn

Let $\mt{D}_n$ be the set of permutations $w\in \mt{S}_n$ such that in the one line notation of $w$, $n+1-i$ appears before $i$ and
there is no number between $i$ and $n+1-i$ that appears between $n+1-i$ and $i$ in the one line notation of $w$.

A recursive description of $\mt{D}_n$ is as follows. In the one line notation of an element $w\in \mt{D}_n$, $n$ always appear to immediate
left of 1. Deleting the sequence `$n1$' from $w$ and then reducing remaining entries by one gives an element of $\mt{D}_{n-2}$.

\begin{table}[htp]
\begin{center}
\begin{tabular}{c|c}\label{Table}
$n $ & $\mt{D}_n$ \\
\hline
1 & $\{[1]\}$\\
2 & $\{[21]\}$ \\
3 & $\{[231], [312]\}$\\
4 & $\{[3241], [3412], [4132]\}$

\end{tabular}
\end{center}
\end{table}

The main result of this section is the following
\bthm\label{thm:W-set inv} The $W$-set of the longest permutation $\pi_0$ is $\mt{D}_n$. In other words,
$$W(\pi_{0,n}) = \mt{D}_n.$$
\ethm

\begin{figure}[htp]
\begin{center}

\begin{tikzpicture}[scale=.47]

\node at (0,0) (a) {$\text{id}$};

\node at (-7.5,5) (b1) {$(1,2)$};
\node at (-2.5,5) (b2) {$(2,3)$};
\node at (2.5,5) (b3) {$(3,4)$};
\node at (7.5,5) (b4) {$(4,5)$};

\node at (-12.5,10) (c1) {$(1,3)$};
\node at (-7.5,10) (c2) {$(1,2)(3,4)$};
\node at (-2.5,10) (c3) {$(1,2)(4,5)$};
\node at (2.5,10) (c4) {$(2,4)$};
\node at (7.5,10) (c5) {$(2,3)(4,5)$};
\node at (12.5,10) (c6) {$(3,5)$};

\node at (-12.5,15) (d1) {$(1,4)$};
\node at (-7.5,15) (d2) {$(1,3)(4,5)$};
\node at (-2.5,15) (d3) {$(1,3)(2,4)$};
\node at (2.5,15) (d4) {$(2,4)(3,5)$};
\node at (7.5,15) (d5) {$(1,2)(3,5)$};
\node at (12.5,15) (d6) {$(2,5)$};

\node at (-5,20) (e1) {$(1,5)$};
\node at (-10,20) (e2) {$(1,4)(2,3)$};
\node at (-0,20) (e3) {$(1,3)(2,5)$};
\node at (10,20) (e4) {$(2,5)(3,4)$};
\node at (5,20) (e5) {$(1,4)(3,5)$};

\node at (-5,25) (f1) {$(1,5)(2,3)$};
\node at (0,25) (f2) {$(1,4)(2,5)$};
\node at (5,25) (f3) {$(1,5)(3,4)$};

\node at (0,30) (g) {$(1,5)(2,4)$};

\node at (-4,2) {$\blue{s_1}$};
\node at (-1.8,2) {$\blue{s_2}$};
\node at (1.7,2) {$\blue{s_3}$};
\node at (4,2) {$\blue{s_4}$};

\node at (-9.5,6) {$\blue{s_2}$};
\node at (-8,6) {$\blue{s_3}$};
\node at (-6.2,5.6) {$\blue{s_4}$};
\node at (-3.7,6) {$\blue{s_1}$};
\node at (-2,6) {$\blue{s_3}$};
\node at (-1,5.2) {$\blue{s_4}$};
\node at (1,5.2) {$\blue{s_1}$};
\node at (2,6) {$\blue{s_2}$};
\node at (3.7,6) {$\blue{s_4}$};
\node at (6.2,5.2) {$\blue{s_1}$};
\node at (8,6) {$\blue{s_2}$};
\node at (9.5,6) {$\blue{s_3}$};

\node at (-13,11) {$\blue{s_3}$};
\node at (-10.6,11) {$\blue{s_4}$};
\node at (-7.1,11) {$\blue{s_2}$};
\node at (-5,10.4) {$\blue{s_4}$};
\node at (-2.5,10.8) {$\blue{s_2}$};
\node at (-0.5,10.4) {$\blue{s_3}$};
\node at (1.5,10.8) {$\blue{s_1}$};
\node at (3.3,10.6) {$\blue{s_4}$};
\node at (5.5,10.4) {$\blue{s_1}$};
\node at (7.2,10.8) {$\blue{s_3}$};
\node at (10.7,11) {$\blue{s_1}$};
\node at (13,11) {$\blue{s_2}$};

\node at (-12.7,16) {$\blue{s_4}$};
\node at (-10.5,15.6) {$\blue{s_2}$};
\node at (-6.2,15.6) {$\blue{s_3}$};
\node at (-3.7,15.8) {$\blue{\{s_1,s_3\}}$};
\node at (-1.5,15.8) {$\blue{s_4}$};
\node at (2.5,16) {$\blue{s_1}$};
\node at (4.2,15.9) {$\blue{\{s_2,s_4\}}$};
\node at (7,15.7) {$\blue{s_2}$};
\node at (10.8,16) {$\blue{s_1}$};
\node at (12.6,16) {$\blue{s_3}$};

\node at (-9.8,21) {$\blue{s_4}$};
\node at (-5.8,21) {$\blue{s_2}$};
\node at (-3.2,20.4) {$\blue{s_3}$};
\node at (-1.2,20.8) {$\blue{s_1}$};
\node at (0.4,21) {$\blue{s_3}$};
\node at (3.5,20.8) {$\blue{s_2}$};
\node at (5.4,21) {$\blue{s_4}$};
\node at (9.6,21) {$\blue{s_1}$};

\node at (-4.8,26) {$\blue{s_3}$};
\node at (-0.15,26) {$\blue{\{s_1,s_4\}}$};
\node at (4.8,26) {$\blue{s_2}$};

\draw[-, very thick, double]  (a) to (b1);
\draw[-, very thick, double]  (a) to (b2);
\draw[-, very thick, double] (a) to (b3);
\draw[-, very thick, double] (a) to (b4);

\draw[-, very thick] (b1) to (c1);
\draw[-, very thick, double] (b1) to (c2);
\draw[-, very thick, double] (b1) to (c3);
\draw[-, very thick] (b2) to (c1);
\draw[-, very thick] (b2) to (c4);
\draw[-, very thick, double] (b2) to (c5);
\draw[-, very thick, double] (b3) to (c2);
\draw[-, very thick] (b3) to (c4);
\draw[-, very thick] (b3) to (c6);
\draw[-, very thick, double] (b4) to (c3);
\draw[-, very thick, double] (b4) to (c5);
\draw[-, very thick] (b4) to (c6);

\draw[-, very thick] (c1) to (d1);
\draw[-, very thick, double] (c1) to (d2);
\draw[-, very thick] (c2) to (d3);
\draw[-, very thick] (c2) to (d5);
\draw[-, very thick] (c3) to (d2);
\draw[-, very thick] (c3) to (d5);
\draw[-, very thick] (c4) to (d1);
\draw[-, very thick] (c4) to (d6);
\draw[-, very thick] (c5) to (d2);
\draw[-, very thick] (c5) to (d4);
\draw[-, very thick, double] (c6) to (d5);
\draw[-, very thick] (c6) to (d6);

\draw[-, very thick] (d1) to (e1);
\draw[-, very thick, double] (d1) to (e2);
\draw[-, very thick] (d2) to (e5);
\draw[-, very thick] (d3) to (e2);
\draw[-, very thick] (d3) to (e3);
\draw[-, very thick] (d4) to (e4);
\draw[-, very thick] (d4) to (e5);
\draw[-, very thick] (d5) to (e3);
\draw[-, very thick] (d6) to (e1);
\draw[-, very thick, double] (d6) to (e4);

\draw[-, very thick, double] (e1) to (f1);
\draw[-, very thick, double] (e1) to (f3);
\draw[-, very thick] (e2) to (f1);
\draw[-, very thick] (e3) to (f1);
\draw[-, very thick] (e3) to (f2);
\draw[-, very thick] (e4) to (f3);
\draw[-, very thick] (e5) to (f2);
\draw[-, very thick] (e5) to (f3);

\draw[-, very thick] (f1) to (g);
\draw[-, very thick] (f2) to (g);
\draw[-, very thick] (f3) to (g);

\end{tikzpicture}

\caption{Reverse weak order on $\mt{I}_5$.}
\label{fig:inv 5}

\end{center}
\end{figure}


Before proving the theorem, we first state and prove two simple computational lemmas that are needed in the proof.

\blem\label{lem:w set lem 1}
\benum
\item Suppose that $1 \leq j < i \leq n$ and let 
$$v = (s_{n-1} s_{n-2} \cdots s_{j+1})(s_1 s_2 \cdots s_{i-1}).$$\label{lem:w set lem 1-1}
\benum
\item If $i = j+1$, then $v \cdot e = (1,n)$ and $\ell_{(n)}(v \cdot e) = \ell(v)$.\label{lem:w set lem 1-1a}
\item If $i = j+2$, then $v \cdot e = (1,n)$ and $\ell_{(n)}(v \cdot e) = \ell(v)-1$.\label{lem:w set lem 1-1b}
\item If $i > j+2$, then $v \cdot e = (1,n)(j+1,i-1)$ and $\ell_{(n)}(v \cdot e) = \ell(v)-1$.\label{lem:w set lem 1-1c}
\eenum
\item Suppose that $1 \leq i < j \leq n$ and let $$v = (s_{n-1} s_{n-2} \cdots s_j)(s_1 s_2 \cdots s_{i-1}).$$
Then $v \cdot e = (1,i)(j,n)$ and $\ell_{(n)}(v \cdot e) = \ell(v)$.\label{lem:w set lem 1-2}
\eenum
\elem

\begin{proof}
We prove part (\ref{lem:w set lem 1-1c}), the proofs of the other statements being similar and left to the reader.
First, a simple inductive argument shows that $$(s_1 s_2 \cdots s_{i-1}) \cdot e = (1,i).$$
Since $i > j+2$, $$s_{j+1} \cdot (1,i) = (1,i)(j+1,j+2).$$
Another straightforward inductive argument shows that
$$(s_{i-2} s_{i-3} \cdots s_{j+1}) \cdot (1,i) = (1,i)(j+1,i-1).$$  Now the key observation is that $i$ occurs before $i-1$ in the
one-line notation of $(1,i)(j+1,i-1)$ (indeed $i$ occurs in the very first position), so
$$s_{i-1} \cdot (1,i)(j+1,i-1) = (1,i)(j+1,i-1).$$  From that point on, a third simple induction gives
$$(s_n s_{n-1} \cdots s_{i}) \cdot (1,i)(j+1,i-1) = (1,n) (j+1,i-1).$$  In each step, $\ell$ and $\ell_{(n)}$
both increase by one, except for the key step, where the length $\ell$ increases by one
while $\ell_{(n)}$ remains the same.  This proves the second statement.
\end{proof}

\blem\label{lem:w set lem 2}
Suppose that $\pi, \pi' \in \mt{I}_n$ are involutions such that the cycle $(1,n)$ does not occur in $\pi$ but does occur in $\pi'$.
If $\pi' = s \cdot \pi$ for some $s \in S$, then $s = s_1$ or $s = s_{n-1}$.
\elem

\begin{proof}
Looking at Definition \ref{def:down act inv}, the only way for $\pi'$ to contain $(1 \,\, n)$ when $\pi$ does not is if $\pi$ contains either
$(1,n-1)$ or $(2,n)$.  In either case, only the action of $s_1$ or $s_{n-1}$ can produce the cycle $(1,n)$.
\end{proof}


\begin{proof}[Proof of Theorem \ref{thm:W-set inv}]
The proof is by induction on $n$, the cases $n=1$ and $n=2$ being trivial.  Let $w \in W(\pi_{0,n})$ and $i = w^{-1}(1)$, $j = w^{-1}(n)$.
We consider two cases.

The first case is $i > j$.  Set
\begin{equation*}\label{eqn:v def}
v = (s_{n-1} s_{n-2} \cdots s_{j+1})(s_1 s_2 \cdots s_{i-1}),
\end{equation*}
and note that $\ell(v) = n + i - j - 2 \geq n - 1$.
Furthermore, $w' = w v$ has length $\ell(w') = \ell(w) - \ell(v)$.  Then $w'(1) = 1$, $w'(n) = n$ and
the one-line notation for $w'$ is obtained from that of $w$ by moving $1$ to the front and $n$ to the end.
For example, if $w = [3 5 6 1 4 2]$, then $v = s_5 s_4 s_1 s_2 s_3$ and $w' = [1 3 5 4 2 6]$.
Moreover $w \cdot e$ is given by Lemma \ref{lem:w set lem 1}(\ref{lem:w set lem 1-1}).  
Using the length calculations of Lemma \ref{lem:w set lem 1}(\ref{lem:w set lem 1-1}), we must have $i = j + 1$ in order for $w$ to be in $W(\pi_{0,n})$.

Consider $\mt{S}_{n-2}$ (resp., $\mt{I}_{n-2}$) as the permutations (resp., involutions) of the set $\{2,3,\dots,n-1\}$, 
or equivalently of the set $\{1,2,\dots,n\}$ that fix $1$ and $n$.
Then $w' \in \mt{S}_{n-2}$ and $w \cdot e = \pi_{0,n} \Leftrightarrow w' \cdot e = \pi_{0,n-2}$.  
Indeed, the condition on the right is equivalent to $w' \cdot (1,n) = (1,n)\pi_{0,n-2} = \pi_{0,n}$.
Since $\ell(w') = \ell(w) - \ell(v) = \ell_{(n)}(\pi_{0,n}) - (n-1) = \ell_{(n-2)}(\pi_{0,n-2})$ (see Remark \ref{rem:rank calc}), 
$w' \in W(\pi_{0,n-2})$ and, by the induction hypothesis, $w' \in \mt{D}_{n-2}$.  It follows that $w \in \mt{D}_n$.

The second case is $i < j$.  Set
\begin{equation*}
v = (s_{n-1} s_{n-2} \cdots s_j)(s_1 s_2 \cdots s_{i-1}).
\end{equation*}
As before, $w = w' v$ with $\ell(w) = \ell(w') + \ell(v)$, $w'(1) = 1$ and $w'(n) = n$.
Additionally, by Lemma \ref{lem:w set lem 1}(\ref{lem:w set lem 1-2}), $v \cdot e = (1,i)(j,n)$.
Since $w' \in \mt{S}_{n-2}$, any reduced decomposition of $w'$ uses only the transpositions $s_2, s_3, \dots, s_{n-2}$.
Then, by Lemma \ref{lem:w set lem 2}, $w \cdot e = w' \cdot (1,i)(j,n)$ cannot contain the cycle $(1,n)$.  
Therefore, $w \cdot e \neq \pi_{0,n}$ and $w \notin W(\pi_{0,n})$.
\end{proof}

\bcor
The maximal chains of the induced partial order on $\mc{I}_n$ are parameterized by the reduced expressions of the elements of $\mt{D}_n$.
\ecor


\section{Weak Order on $\mu$-involutions}\label{sec:weak order mu-strings}

Fix a positive integer $n$ and let $\mu = (\mu_1, \mu_2, \dots, \mu_k)$ be a composition of $n$.
Let $\nu_j = \sum_{i = 1}^j \mu_i$.  By convention, set $\nu_0 = 0$.
Given a permutation $\pi \in \mt{S}_n$, define the $i$-th $\mu$-string of $\pi$ to be $\alpha_i = [\pi(\nu_{i-1} + 1), \pi(\nu_{i-1} + 2), \dots, \pi(\nu_i)]$.  We write $\pi = [\alpha_1 | \alpha_2 | \cdots | \alpha_k]$ to indicate the $\mu$-strings of $\pi$.

\begin{Convention}\label{conv:string perm}
Given any string $\alpha$ containing each element of an alphabet $\mathbb{A} \subset [n]$ in exactly one position,
we interpret $\alpha$ as the one-line notation of a permutation of $\mathbb{A}$.  (We order $\mathbb{A}$ in increasing order.)
For example, the string $\alpha = [5264]$ is interpreted as the permutation $2 \mapsto 5, 4 \mapsto 2, 5 \mapsto 6, 6 \mapsto 4$.
\end{Convention}

Note that a permutation $\pi \in \mt{S}_n$ is a $\mu$-involution if each of the permutations associated to the $\mu$-strings of $\pi$ is an involution.
Let $\mt{I}_{\mu}$ denote the set of $\mu$-involutions.

\brem\label{rem:n-involutions are involutions}
When $\mu$ is the composition consisting of the single part $n$, then a $\mu$-involution is an ordinary involution in $\mt{S}_n$.
Thus, the notation $\mt{I}_n$ is unambiguous.
\erem

Recall the action of $M(\mt{S}_n)$ on $\mt{I}_\mu$:
\bdfn\label{def:down act mu-inv}
If $\pi^{-1}(i) > \pi^{-1}(i+1)$, then $s_i \cdot \pi = \pi$.  If $\pi^{-1}(i) < \pi^{-1}(i+1)$ and the values $i$ and $i+1$ occur
in different $\mu$-strings of $\pi$, then $s_i \cdot \pi = s_i \pi$.  Otherwise, if $\pi^{-1}(i) < \pi^{-1}(i+1)$ and the values
$i$ and $i+1$ occur in the same $\mu$-string $\alpha_j$ of $\pi$, then the action splits into two further subcases.
If $\alpha_j$ fixes both $i$ and $i+1$ (when the string $\alpha_j$ is viewed as a permutation as in Convention \ref{conv:string perm}),
then $s_i \cdot \pi = [\alpha_1| \dots| \alpha_{j-1}| s_i \alpha_{j}| \alpha_{j+1}| \dots| \alpha_k]$.
Otherwise, $s_i \cdot \pi = [\alpha_1| \dots| \alpha_{j-1}| s_i \alpha_j s_i| \alpha_{j+1}| \dots |\alpha_k]$.
\edfn

\blem\label{lem:down act basics mu-inv}
\benum
\item The operator $\cdot$ in Definition \ref{def:down act mu-inv} extends to the Richardson-Springer monoid of $\mt{S}_n$.
That is\label{lem:da mu 1}
    \benum
    \item $s_i \cdot (s_i \cdot \pi) = s_i \cdot \pi$;\label{lem:da mu 1a}
    \item $s_i \cdot (s_j \cdot \pi) = s_j \cdot (s_i \cdot \pi)\,$ if $\,|j - i| > 1$;\label{lem:da mu 1b}
    \item $s_i \cdot (s_{i+1} \cdot (s_i \cdot \pi)) = s_{i+1} \cdot (s_i \cdot (s_{i+1} \cdot \pi))$.\label{lem:da mu 1c}
    \eenum
\item If $\pi^{-1}(i) > \pi^{-1}(i+1)$, then $\ell_\mu(s_i \cdot \pi) = \ell_\mu(\pi)$.  If $\pi^{-1}(i) < \pi^{-1}(i+1)$,
then $\ell_{\mu}(s_i \cdot \pi) = \ell_{\mu}(\pi) + 1$.\label{lem:da mu 2}
\eenum
\elem


\bdfn\label{def:weak order mu-inv}

The {\em reverse weak order} on $\mt{I}_{\mu}$ is defined by $\pi \leq \pi' \Leftrightarrow \pi' = w \cdot \pi$ for some $w \in M$.
It has a unique minimal element $e = [1 \,\, 2 \dots n]$ and a unique maximal element $\pi_{0,\mu}$.
The maximal element $\pi_{0,\mu}$ is characterized by the fact that the every element of the alphabet of the $i$-th $\mu$-string
is larger than every element of the $(i+1)$-st $\mu$-string and the permutation associated to any
$\mu$-string is the usual longest element which reverses the order of the numbers.  For example, $\pi_{0,(4,2)} = [6543 | 21]$.

\edfn

The reverse weak order on $\mt{I}_{(3,1)}$ is shown in Figure \ref{fig:mu-inv 3,1}.  
(See Section \ref{sec:geom} for an explanation of the double edges.)
The 11 maximal chains correspond to the reduced decompositions of the two permutations $[342|1]$ and $[423|1]$ which make up
$\mt{D}_{(3,1)}$ (see Definition \ref{def:D-set mu-inv} and Theorem \ref{thm:W-set mu-inv}).

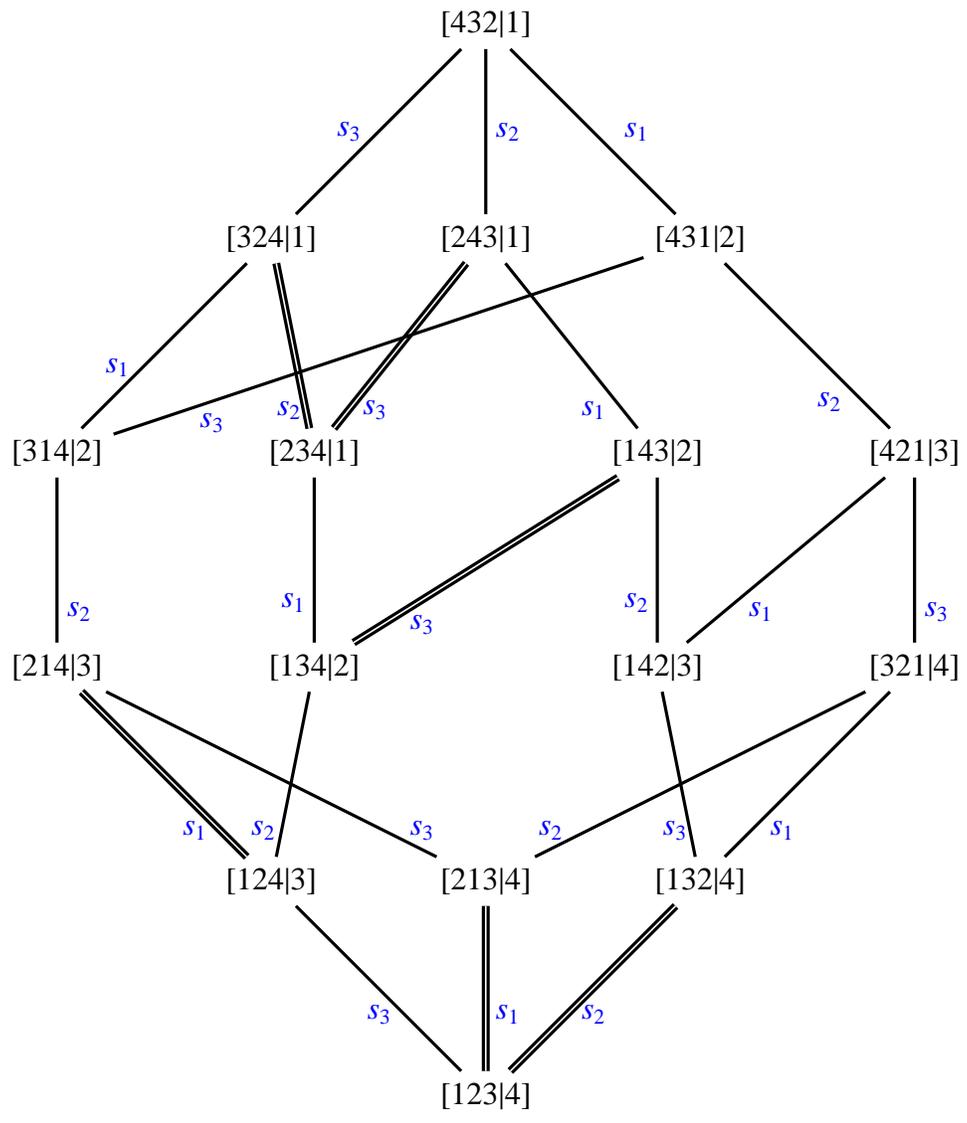
\begin{figure}[htp]
\begin{center}

\begin{tikzpicture}[scale=.57]

\node at (0,0) (a) {$[432|1]$};

\node at (-5,-5) (b1) {$[324|1]$};
\node at (0,-5) (b2) {$[243|1]$};
\node at (5,-5) (b3) {$[431|2]$};

\node at (-10,-10) (c1) {$[314|2]$};
\node at (-4,-10) (c2) {$[234|1]$};
\node at (4,-10) (c3) {$[143|2]$};
\node at (10,-10) (c4) {$[421|3]$};

\node at (-10,-15) (d1) {$[214|3]$};
\node at (-4,-15) (d2) {$[134|2]$};
\node at (4,-15) (d3) {$[142|3]$};
\node at (10,-15) (d4) {$[321|4]$};

\node at (-5,-20) (e1) {$[124|3]$};
\node at (0,-20) (e2) {$[213|4]$};
\node at (5,-20) (e3) {$[132|4]$};

\node at (0,-25) (g) {$[123|4]$};

\node at (-3.2,-2.5) {$\blue{s_3}$};
\node at (3.5,-2.5) {$\blue{s_1}$};
\node at (.5,-2.5) {$\blue{s_2}$};

\node at (-8.6,-8) {$\blue{s_1}$};
\node at (-6.4,-9.3) {$\blue{s_3}$};

\node at (-4.6,-9) {$\blue{s_2}$};
\node at (-2.6,-9) {$\blue{s_3}$};
\node at (2.5,-9) {$\blue{s_1}$};
\node at (8,-8.8) {$\blue{s_2}$};

\node at (-9.5,-13.7) {$\blue{s_2}$};
\node at (-4.5,-13.5) {$\blue{s_1}$};
\node at (-1.5,-14) {$\blue{s_3}$};
\node at (3.5,-13.5) {$\blue{s_2}$};
\node at (6.4,-13.7) {$\blue{s_1}$};
\node at (10.5,-13.7) {$\blue{s_3}$};

\node at (-6.8,-18.8) {$\blue{s_1}$};
\node at (-5.2,-18.8) {$\blue{s_2}$};
\node at (-1.5,-18.8) {$\blue{s_3}$};
\node at (1.5,-18.8) {$\blue{s_2}$};
\node at (4.4,-18.8) {$\blue{s_3}$};
\node at (6.9,-18.8) {$\blue{s_1}$};

\node at (-2.5,-23.1) {$\blue{s_3}$};
\node at (2.5,-23.1) {$\blue{s_2}$};
\node at (.5,-23.1) {$\blue{s_1}$};

\draw[-, very thick]  (a) to (b1);
\draw[-, very thick]  (a) to (b2);
\draw[-, very thick] (a) to (b3);

\draw[-, very thick] (b1) to (c1);
\draw[-, very thick, double] (b1) to (c2);
\draw[-, very thick, double] (b2) to (c2);
\draw[-, very thick] (b2) to (c3);
\draw[-, very thick] (b3) to (c1);
\draw[-, very thick] (b3) to (c4);

\draw[-, very thick] (c1) to (d1);
\draw[-, very thick] (c2) to (d2);
\draw[-, very thick, double] (c3) to (d2);
\draw[-, very thick] (c3) to (d3);
\draw[-, very thick] (c4) to (d3);
\draw[-, very thick] (c4) to (d4);

\draw[-, very thick, double] (d1) to (e1);
\draw[-, very thick] (d1) to (e2);
\draw[-, very thick] (d2) to (e1);
\draw[-, very thick] (d3) to (e3);
\draw[-, very thick] (d4) to (e3);
\draw[-, very thick] (d4) to (e2);

\draw[-, very thick] (e1) to (g);
\draw[-, very thick, double] (e2) to (g);
\draw[-, very thick, double] (e3) to (g);

\end{tikzpicture}

\caption{Reverse weak order on $\mt{I}_{3,1}$.}
\label{fig:mu-inv 3,1}

\end{center}
\end{figure}

\begin{Lemma}\label{def:quad length}
Let $\mt{S}_\mu$ denote the parabolic subgroup corresponding to $\mu$.
Then the rank function $\ell_\mu$ on $\mt{I}_\mu$ is given by
\begin{align}\label{A:length in I_mu}
\ell_\mu(\pi) := \min_{w \in \mt{S}_\mu} \ell(w \pi) + \sum_{i=1}^k \frac{\ell(\alpha_i) + \text{exc}(\alpha_i)}{2},
\end{align}
where $\pi = [\alpha_1 | \dots | \alpha_k]\in \mt{I}_\mu$.
\end{Lemma}

\begin{proof}
The value of the rank function $\ell_\mu(\pi)$ is equal to the
length of a maximal chain from $e$ to $\pi$. We construct such a chain by the action of $M(\mt{S}_n)$ in two steps.
First, we permute the entries of $\pi$ in such a way that if $n - \nu_i < j \leq n - \nu_{i-1}$, then $j$ appears in the $i$-th $\mu$ string.
The minimal number of steps (simple tranpositions to act with) required for this rearrangement is $\min_{w \in \mt{S}_\mu} \ell(w \pi) $.
In the second step we reorder the entries of each $\mu$-string. The number of steps required is the summation on the right hand side
of (\ref{A:length in I_mu}).
\end{proof}

\bdfn\label{def:W-set mu-inv}
Let $\pi \in \mt{I}_{\mu}$.  The $W$-set of $\pi$ is
$$
W(\pi) := \{ w \in \mt{S}_n : w \cdot e = \pi \text{ and } \ell(w) = \ell_\mu(\pi) \}.
$$
\edfn

\bdfn\label{def:D-set mu-inv}
Define $\mt{D}_{\mu}$ to consist of permutations $w \in \mt{S}_n$ whose $i$-th $\mu$-string uses the alphabet consisting of
integers $j$ such that $n - \nu_i < j \leq n - \nu_{i-1}$ and has its associated permutation in $\mt{D}_{\mu_i}$.
\edfn

\bexam\label{ex:D-set mu-inv}
$\mt{D}_{(4,2)} = \{ [5463|21], [5634|21], [6354|21] \}$.
\eexam

\bthm\label{thm:W-set mu-inv}
$$W(\pi_{0,\mu}) = \mt{D}_{\mu}.$$
\ethm

\begin{proof}
Recall that $\mt{S}_\mu \cong \mt{S}_{\mu_1} \times \mt{S}_{\mu_2} \times \cdots \times \mt{S}_{\mu_k}$ 
is the parabolic subgroup generated by $s_i$ for $i \in I$ and $\mt{S}^\mu$ is the set of minimal length {\em right} coset representatives.  
Suppose $w \in W(\pi_{0,\mu})$ and write
$w = uv$ with $u \in \mt{S}_\mu$ and $v \in \mt{S}^\mu$.

We first claim that $v \cdot e$ must result in a $\mu$-involution such that
\begin{equation*}\label{eqn:mu-string cond}\tag{*}
\text{the } i\text{-th}\ \mu\text{-string consists of integers } j \text{ such that } n - \nu_i < j \leq n - \nu_{i-1}.
\end{equation*}
Indeed, we may write $u = u_1 u_2 \cdots u_k$ with each $u_i \in \mt{S}_{\mu_i} \subset \mt{S}_{\mu}$.  
The $u_i$'s commute with each other and the effect of $u_i$ on any $\mu$-involution is to only change the $i$-th $\mu$-string.  
In particular, if $\pi$ is any $\mu$-involution, the alphabets of each $\mu$-string must be the same for $\pi$ and $u \cdot \pi$.  
Since $\pi_{0,\mu}$ satisfies \eqref{eqn:mu-string cond}, so must $v \cdot e$.

Consider $v_{\max}$, the element of $S^{\mu}$ of maximal length.  Explicitly, $v_{\max}$ satisfies \eqref{eqn:mu-string cond} 
and each $\mu$-string is written in increasing order.  
At the same time, $v_{\max}$ is a $\mu$-involution and it has the minimal value of $\ell_{\mu}$ among all $\mu$-involutions 
satisfying \eqref{eqn:mu-string cond}.  Moreover, $v_{\max} \cdot e = v_{\max}$ and $\ell_{\mu}(v_{\max}) = \ell(v_{\max})$.  
Therefore, $v \cdot e$ cannot satisfy \eqref{eqn:mu-string cond} unless $v = v_{\max}$.

Then $u \cdot v_{\max} = \pi_{0,\mu} \Leftrightarrow u_i \cdot e = \pi_{0,\mu_i}$ for each $1 \leq i \leq k$.  
Here, we are interpreting $u_i$ as a permutation of the set of integers $j$ with $n - \nu_i < j \leq n - \nu_{i-1}$.  
In order for $w \cdot e = \pi_{0,\mu}$ with $\ell(w)$ minimal, each $u_i$ must be an element of $W(\pi_{0,\mu_i}) = \mt{D}_{\mu_i}$, 
shifted appropriately so as to become a permutation of the alphabet of the $i$-th $\mu$-string.  
The definition of $\mt{D}_{\mu}$ is equivalent to the set of elements $u_1 u_2 \cdots u_k v_{\max}$ so obtained.

\end{proof}

\section{Geometric Consequences}\label{sec:geom}

We briefly describe the geometric implications of our theorems.  
The closed orbit of $\ms{X}_n$, corresponding to the partition $\mu = (1,1,\dots,1)$ ($n$ $1$'s) is isomorphic to the complete 
flag variety $G / B^-$ and the inclusion $i : G / B^{-} \hookrightarrow \ms{X}$ induces a restriction map on the cohomology 
$i^* : H^*(\ms{X};\Z) \rightarrow H^*(G/B^{-};\Z)$.  For each $B$-stable subvariety $Y \subseteq \ms{X}$, 
it is natural to ask for a description of the class $i^*([Y])$ in the Schubert basis of $H^*(G/B^{-};\Z)$.  
Brion gives such a description in terms of the weak order on $B(\ms{X})$, the collection of $B$-stable subvarieties of $\ms{X}$.  
Given $w \in S_n$ and any reduced decomposition $w = s_{i_1} s_{i_2} \cdots s_{i_l}$, $w \cdot Y$ is defined to be 
$P_{s_{i_1}} P_{s_{i_2}} \cdots P_{s_{i_l}} Y$, where $P_{s_j} = BW_{s_j}B$ is the minimal parabolic subgroup associated to $s_j$.  
A covering relation $s_i \cdot Y = Y'$ is depicted by drawing a directed edge from $Y$ to $Y'$ with label $s_i$. 
The edge is single (resp., double) if the morphism $P_{s_i} \times_B Y \rightarrow Y'$ has degree $1$ (resp., $2$).

If the smallest $G$-stable subvariety of $\ms{X}$ containing $Y$ is $\ms{X}_{\mu}$, the $W$-set of $Y$ is defined to be
$$
W(Y) = \{ w \in S_n : w \cdot Y = \ms{X}_{\mu} \text{ and } \ell(w) = \text{codim}(Y,\ms{X}_{\mu}) \}.
$$
Associated to each path in the associated weak order of $B(\ms{X})$ is the number of double edges in the path.  
In \cite{Brion98}, Brion shows that the number $D(Y)$ of such double edges is the same for any path from $Y$ to $\ms{X}_{\mu}$ and that
$$i^*(Y) = 2^{D(Y)} \sum_{w \in W(Y)} [B w B^{-} / B^{-}].$$

If $Y = Y_{\pi}$ is associated to the $\mu$-involution $\pi$, then chains in the weak order from $Y_{\pi}$ to $\ms{X}_{\mu}$ 
correspond to chains in the reverse weak order from $e$ to $\pi$.  Moreover, a double edge in the reverse weak order graph occurs 
if and only $\pi \rightarrow s_i \cdot \pi = s_i \pi$ corresponds to the case where $i$ and $i+1$ are contained in the same $\mu$-string and 
$\pi$ fixes both $i$ and $i+1$.  Such double edges correspond to adding a $2$-cycle in the cycle decomposition 
of the corresponding $\mu$-string.  Hence $D(\pi) := D(Y_{\pi})$ is equal to the number of involutions occurring in the $\mu$-strings of $\pi$.

Therefore, when $\pi = \pi_{0,\mu}$, 
$$
D(\pi_{0,\mu}) = \sum_{i=1}^k \lfloor \frac{\mu_i}{2} \rfloor.
$$
Moreover, maximal chains from $e$ to $\pi$ correspond to chains from $Y_{\pi}$ to $\ms{X}_{\mu}$ traversed in the opposite direction, 
so $W(Y_{\pi}) = W(\pi)^{-1}$.
It follows from Theorems \ref{thm:W-set inv} and \ref{thm:W-set mu-inv} that
$$
i^*([Y_{\pi_{0,\mu}}]) = 2^{D(\pi_{0,\mu})} \sum_{w \in \mt{D}_{\mu}} \left[ B w^{-1} B^{-} / B^{-} \right].
$$

Standard properties of geometric divided difference operators allow one to compute general classes $i^*([Y])$ 
analogous to how one computes arbitrary Schubert polynomials knowing the formula 
$\mathfrak{S}_{w_{0,n}} = x_1^{n-1} x_2^{n-2} \cdots x_{n-1}$.

In the case of ordinary involutions, based on numerical evidence for $n \leq 10$ and some multidegree calculations, 
we conjecture that $i^*([Y_{\pi_{0,n}}])$ also has a similar factorization:
$$
i^*([Y_{\pi_{0,n}}]) = \prod_{1 \leq i \leq j \leq n-i} (x_i + x_j),
$$
where $x_i$ is the first Chern class of the dual of the $i$-th quotient line bundle in the tautological flag on $G / B^{-}$.
Viewed in an appropriate geometric context, we hope that such a ``product = sum'' formula can be proved uniformly in $n$, 
thereby yielding rather interesting factorizations of certain sums of Schubert polynomials.  We plan to return to this idea in future work.

\bibliography{References}
\bibliographystyle{plain}

\end{document}